\pdfoutput=1
\documentclass[12pt]{amsart}
\usepackage{geometry}
\usepackage{amssymb, amsmath, amsthm, color, tikz,enumerate,mathrsfs}
\usepackage{pdfpages}
\usepackage{hyperref}
\usepackage{textcomp}
\usepackage{color}
\usepackage{xcolor}
\usepackage{listings}
\usepackage[foot]{amsaddr}

\usepackage[all]{xy}

\usepackage{times}

\usepackage{mathrsfs}

\usepackage{enumitem}

\usepackage{url}
\usepackage{hyperref}

\usepackage{fullpage}

\newtheorem{prop}[subsection]{Proposition}
\newtheorem{lemma}[subsection]{Lemma}
\newtheorem{thm}[subsection]{Theorem}
\newtheorem{coro}[subsection]{Corollary}

\newtheorem{uthm}{Theorem}

\theoremstyle{definition}

\newtheorem{defi}[subsection]{Definition}
\newtheorem{rmk}[subsection]{Remark}
\newtheorem{ex}[subsection]{Example}

\numberwithin{equation}{subsubsection}

\renewcommand{\mathcal}{\mathscr}

\newcommand{\Nat}{\mathbb{N}}

\newcommand{\Z}{\mathbb{Z}}

\renewcommand{\H}{\mathrm{H}}

\newcommand{\GL}{{\bf {GL}}}

\newcommand{\Of}{\mathcal{O}}

\newcommand{\Sgoth}{\mathfrak{S}}

\newcommand{\ab}{\mathrm{ab}}

\newcommand\Cf{\mathcal{C}}
\DeclareMathOperator{\Char}{Char}

\DeclareMathOperator{\Det}{Det}

\newcommand{\fl}{\rightarrow}
\newcommand{\ra}{\rightarrow}

\def\flnom#1{\stackrel{#1}{\fl}}

\DeclareMathOperator{\Frac}{Frac}

\newcommand{\id}{\mathrm{id}} 

\newcommand{\iso}{\stackrel{\sim}{\fl}}
\DeclareMathOperator{\Ker}{Ker}

\DeclareMathOperator{\LPC}{LPC}
\newcommand{\Mcal}{\mathcal{M}} 

\mathchardef\mhyphen="2D
\newcommand\mfrak{\mathfrak{m}}

\DeclareMathOperator{\Ob}{\mathrm{Ob}}

\DeclareMathOperator{\PChar}{PChar}

\newcommand\quash[1]{}

\newcommand\Set{\mathbf{Set}}
\DeclareMathOperator\sgn{sgn}
\DeclareMathOperator{\Spec}{Spec}

\DeclareMathOperator\Tor{Tor}

\DeclareMathOperator\Tr{Tr}

\newcommand{\univ}{\mathrm{univ}} 

\makeatletter
\newcommand{\setword}[2]{%
  \phantomsection
  #1\def\@currentlabel{\unexpanded{#1}}\label{#2}%
}
\makeatother

\makeatletter
\@namedef{subjclassname@2020}{%
  \textup{2020} Mathematics Subject Classification}
\makeatother

\title{Comparison of different definitions of pseudocharacters}
\date{\today.}
\author{Kathleen Emerson and Sophie Morel}

\address{CNRS et Unit\'e De Math\'ematiques Pures Et Appliqu\'ees,
ENS de Lyon,
69342 Lyon Cedex 07,
France.\hfill\break
\tt http://perso.ens-lyon.fr/sophie.morel/,
sophie.morel@ens-lyon.fr.}

\subjclass[2020]
{Primary
14M35, 
13A50; 
Secondary
14L24, 
14R20 
}

\keywords{
pseudorepresentations,
pseudocharacters,
character variety,
invariant theory
}

\begin{document}

\begin{abstract}
We prove that the definitions of a $d$-dimensional pseudocharacter
(or pseudorepresentation) given by Chenevier and V. Lafforgue agree over
any ring. We also compare the scheme of Lafforgue's $G$-pseudocharacters
of a group with its $G$-character variety.

\end{abstract}

\maketitle

Pseudocharacters (also known as pseudorepresentations) were first
introduced by Wiles (\cite{Wiles}) and Taylor (\cite{Taylor}) to
study congruences modulo power of prime numbers
between representations of Galois groups.
They defined congruence between representations as congruence between
their characters, and this led them to the introduction of
pseudocharacters, which are functions on a group whose properties
mimick those of the character of a finite-dimensional representation.
More precisely, let $\Gamma$ be a group, $A$ be a commutative unital
ring and $d$ be a positive integer. A map $T:\Gamma\ra A$ is a
\emph{$d$-dimensional pseudocharacter} if $d!\in A^\times$,
$T(1)=d$, $T(\gamma_1\gamma_2)=
T(\gamma_2\gamma_1)$ for all $\gamma_1,\gamma_2\in\Gamma$ and $T$
satisfies the $d$-dimensional \emph{pseudo-character identity}:
for all $\gamma_1,\ldots,\gamma_{d+1}\in\Gamma$,
\[\sum_{\sigma\in\Sgoth_{d+1}}\sgn(\sigma)T^\sigma(\gamma_1,\ldots,\gamma_{d+1})=0,\]
where, if the decomposition into cycles of $\sigma\in\Sgoth_{d+1}$
(including trivial cycles) is
\[\sigma=(i_{1,1}i_{1,2}\ldots i_{1,n_1})\ldots(i_{k,1}i_{k,2}\ldots i_{k,n_k}),\]
then
\[T^\sigma(\gamma_1,\ldots,\gamma_{d+1})=T(\gamma_{i_{1,1}}\gamma_{i_{1,2}}\ldots
\gamma_{i_{1,n_1}})\ldots T(\gamma_{i_{k,1}}\gamma_{i_{k,2}}\ldots
\gamma_{i_{k,n_n}}).\]

We denote by $\PChar_{\Gamma,d}(A)$ the set of $d$-dimensional pseudo-characters
on $\Gamma$ with values in $A$. This is a contravariant functor from
the category of commutative $\Z[1/d!]$-algebras to that of sets,
and it is not hard to see that it is representable, unlike the functor
sending $A$ to the set 
of equivalence classes of $d$-dimensional
representations $\Gamma\ra\GL_d(A)$. Moreover, if $A$ is an algebraically
closed field (in which $d!$ is invertible), then Taylor has shown 
in \cite[Theorem~1]{Taylor} that the trace
induces a bijection from the set of equivalence classes of $d$-dimensional
\emph{semi-simple} representations of $\Gamma$ over $A$ and
$\PChar_{\Gamma,d}(A)$ and, if
$k$ is a field of characteristic $0$, then Chenevier has shown in
\cite[Proposition~2.3] {Chenevier2}
that $\PChar_{\Gamma,d,\Spec k}$ is naturally
isomorphic to the $\GL_d$-character variety of $\Gamma$ over $\Spec k$.

The original definition of a $d$-dimensional
pseudocharacter does not work well in characteristic $p$ dividing $d!$,
because a semi-simple $d$-dimensional representation is no longer determined
by its trace. However, it is determined if we know all the coefficients of
its characteristic polynomial, and Chenevier (\cite{Chenevier1}) came up
with a compact way of packaging all that data in his theory of
determinants. Determinants make sense over an arbitrary ring $A$ without
the requirement that $d!\in A^\times$, and the functor $\Det_{\Gamma,d}$ that
they define is representable. Moreover, if $\rho:\Gamma\ra\GL_d(A)$ is a
morphism, then it defines a determinant $D_\rho\in\Det_{\Gamma,d}(A)$.
Chenevier shows that, if
$k$ is an algebraically closed field, then $\rho\mapsto D_\rho$ is a bijection
from the set of $d$-dimensional semi-simple representations
of $\Gamma$ over $k$ to $\Det_{\Gamma,k}(k)$.
However, it is not known whether the determinant
functor $\Det_{\Gamma,d}$ is isomorphic to the $\GL_d$-character variety of
$\Gamma$.

Vincent Lafforgue also gave a very general definition of pseudocharacters,
that this time makes sense for representations with values in any
affine group scheme $G$ over a base ring $C$. 
The definition of his pseudocharacters is
less compact, as it requires data for every invariant function on any
power $G$. Again, we obtain a contravariant
functor $\LPC_{\Gamma,G}$ on the category of commutative $C$-algebras, 
which turns out to be representable by an affine scheme.
Also, if $\rho:\Gamma\ra G(A)$ is a morphism, then
it defines a Lafforgue pseudocharacter $\Theta_\rho\in\LPC_{\Gamma,G}(A)$.
Lafforgues proves in \cite[Proposition~11.7]{Lafforgue1}
(see also \cite[Proposition~8.2]{Lafforgue2}) 
that, if $A=k$ is an algebraically closed field,
then $\rho\mapsto\Theta_\rho$ is a bijection
from the set of $G(k)$-conjugacy classes of $G$-completely
reducible morphisms $\Gamma\ra G(k)$ to $\LPC_{\Gamma,G}(k)$;
see also Theorem~4.5 of the
paper \cite{BHKT} of B\"ockle, Harris, Khare and Thorne
for the details of the proof in positive characteristic, where they also
give the more general definition of a Lafforgue pseudocharacter
(Lafforgue himself was working over a field).

In Section~2, we construct a is morphism of schemes from
$\LPC_{\Gamma,G}$ to the $G$-character variety of $\Gamma$ compatible with
the bijections of the previous paragraph. We prove that, for
$G$ a geometrically reductive group scheme with connected
geometric fibers over a Dedekind domain or a field, this morphism is
an \emph{adequate homeomorphism} in the sense of Alper
(see \cite[Definition~3.3.1]{Alper}), that is, an integral
universal homeomorphism which is
a local isomorphism at all points with residue characteristic $0$;
see Proposition~\ref{prop_charvar}.

We come back to the case $G=\GL_d$. The main result
of this note is that, over any ring, Lafforgue pseudocharacters of
$\Gamma$ are in bijection with $d$-dimensional determinants. More precisely,
we prove the following theorem:

\begin{uthm}(See Theorem~\ref{thm_main}.)
There exists an isomorphism $\alpha:\LPC_{\Gamma,\GL_d}\ra\Det_{\Gamma,d}$ such
that, for any commutative ring $A$ and any morphism $\rho:\Gamma\ra\GL_d(A)$,
$\alpha$ sends $\Theta_\rho$ to $D_\rho$.

\end{uthm}

The plan of the paper is as follows. In Section~\ref{Det}, we recall
Chenevier's definition of determinants and the properties of determinants
that we will need; in particular, we give Vaccarino's
formula for the universal determinant ring of a free monoid (see
\cite{Vaccarino} and \cite[Theorem~1.15]{Chenevier2}).
In Section~\ref{LPC}, we define Lafforgue pseudocharacters
and give some of their basic properties, in particular the (easy) fact
that they define a representable functor and the relationship with the
character variety (Proposition~\ref{prop_charvar}).
We construct the morphism
$\alpha$ in Section~\ref{LPCDet} and, in Section~\ref{DetLPC}, we prove that
this morphism is invertible. The main ingredients are 
Donkin's (\cite{Donkin}) description of
generators of $\Of(\GL_{d,\Z}^n)^{\GL_{d,\Z}}$ and the result of Vaccarino
recalled in Section~\ref{Det}.

We thank Julian Quast and Olivier Ta\"ibi for useful discussions.

\section{Determinants of algebras}
\label{Det}

In this section, we recall the definition of determinants by
Chenevier (\cite{Chenevier2}) and some basic properties of the
functor of determinants. Throughout this section, $A$ denotes a commutative
unital ring, and $A$-algebras are always supposed to be associative and
unital.
Let $\Cf_A$ be the category of commutative
$A$-algebras and $\Set$ be the category of sets.
We first recall Roby's definition of a polynomial law (\cite{Roby}, I.2),
a homogeneous polynomial law (\cite{Roby}, I.8) and a multiplicative
polynomial law (\cite{Roby2}, III.4).

\begin{defi}
Let $M$ and $N$ be $A$-modules. A \emph{polynomial law} $f:M\ra N$ is
a morphism of functors from $\Cf_A$ to $\Set$ between the functors
$B\mapsto M\otimes_A B$ and $B\mapsto N\otimes_A B$. For every object
$B$ of $\Cf_A$, we denote the resulting map $M\otimes_A B\ra N\otimes_A B$
by $f_B$.

\end{defi}

\begin{defi}
Let $M$ and $N$ be $A$-modules,
and let $f:M\ra N$ be a polynomial law. Let $d\in\Nat$.
\begin{itemize}
\item[(1)] We say that $f$ is \emph{homogeneous of degree $d$} if,
for all $B\in\Cf_A$, $x\in M\otimes_A B$ and $b\in B$, we have
\[f_B(bx)=b^d f_B(x).\]
\item[(2)] Suppose that $M$ and $N$ are $A$-algebras. We say that $f$
is \emph{multiplicative} if, for every $B\in\Cf_A$, we have $f_B(1)=1$ and
$f_B(xy)=f_B(x)f_B(y)$ if $x,y\in M\otimes_A B$.

\end{itemize}

\end{defi}

If $M$ and $N$ are $A$-algebras and $d\in\Nat$, we write
$\Mcal_A^d(M,N)$ for the set of multiplicative polynomial laws from $M$ to $N$
that are homogeneous of degree $d$.

The following definition is due to Chenevier (see \cite{Chenevier2}, 1.5).

\begin{defi}
Let $R$ be an $A$-algebra. An \emph{$A$-valued $d$-dimensional determinant
on $R$} is an element of $\Mcal_A(R,A)$. When $R=A[\Gamma]$ for some group
$\Gamma$, we also talk about
\emph{$A$-valued $d$-dimensional determinants on $\Gamma$}.

\end{defi} 

We write $\Det_{R,d}(A)$ (resp. $\Det_{\Gamma,d}(A)$) for the set of
$A$-valued $d$-dimensional determinants on $R$ (resp. $\Gamma$);
so $\Det_{\Gamma,d}(A)=\Det_{A[\Gamma],d}(A)$.

\begin{ex}\label{ex_det}
\begin{itemize}
\item[(1)] Let $X$ be a set, and write $A\{X\}$ for the free unital associative
$A$-algebra on $X$. Any map $\rho:X\ra M_d(A)$ gives rise to an $A$-valued
$d$-dimensional determinant $D_\rho$ on $A\{X\}$ as follows. For any object
$B$ of $\Cf_A$, the map $\rho$ induces a morphism of $B$-algebras
$\rho_B:B\{X\}\ra M_d(B)$, where $M_d$ is the scheme
of $d\times d$ matrices; moreover, if $B\ra B'$ is a morphism of
commutative $A$-algebras, then we get a commutative diagram:
\[\xymatrix{B\{X\}\ar[d]\ar[r]^-{\rho_B} & M_d(B)\ar[d] \\
B'\{X\}\ar[r]_-{\rho_{B'}} & M_d(B')}\]
We define $D_{\rho,B}:B\{X\}\ra B$ by $D_{\rho,B}(\omega)=\det(\rho_B(\omega))$,
where $\det$ is the usual determinant on $M_d(B)$.

\item[(2)] Let $\Gamma$ be a group and $\rho:\Gamma\ra\GL_d(A)$ be a morphism
of groups. Then we get an $A$-valued $d$-dimensional determinant
$D_\rho$ on $A[\Gamma]$ in a similar way: for every commutative
$A$-algebra $B$ and
every $x\in B[\Gamma]$, we set $D_{\rho,B}=\det(\rho_B(x))$, where
$\rho_B:B[\Gamma]\ra M_d(B)$ sends an element $\sum_{i=1}^n b_i\gamma_i$ of
$B[\Gamma]$, with $b_i\in B$ and $\gamma_i\in\Gamma$, to
$\sum_{i=1}^n b_i\rho(\gamma_i)$.

\end{itemize}
\end{ex}

\begin{ex}\label{ex_Lambda}
Let $R$ be an $A$-algebra and $D\in\Mcal_A^d(R,A)$. Following
Chenevier (\cite{Chenevier2}, 1.10), we can define
polynomial laws $\Lambda_i:R\ra A$, for $0\leq i\leq d$, as follows. 
For any commutative $A$-algebra $B$ and any $x\in R\otimes_A B$, set
\[D_{B[T]}(T-x)=\sum_{i=0}^d\Lambda_{i,B}(x)T^{d-i}.\]
Note that $\Lambda_i$ is homogeneous of degree $i$.
If $D=D_\rho$ for $\rho$ as in Example~\ref{ex_det}(1) or (2), then
$\Lambda_{i,B}(x)$ is the coefficient of the degree $d-i$ term in the
characteristic polynomial of the matrix $\rho_B(x)\in M_d(B)$.

\end{ex}

\begin{defi}
Let $R$ be an $A$-algebra and $d\in\Nat$.
The \emph{determinant functor} $\Det_{R,d}:\Cf_A\ra\Set$ is defined by
\[\Det_{R,d}(B)=\Mcal_B^d(R\otimes_A B,B)=\Mcal_A^d(R,B)\]
(the equality $\Mcal_B^d(R\otimes_A B)=\Mcal_A^d(R,B)$ is proved in
\cite[3.4]{Chenevier2}).

\end{defi}

If $R=A[\Gamma]$ with $\Gamma$ a group, we also write $\Det_{\Gamma,d}$ instead
of $\Det_{R,d}$.

\begin{rmk}
If $D_1\in\Det_{R,d_1}$ and $D_2\in\Det_{R,d_2}$, we can define
$D_1\times D_2\in\Det_{R,d_1+d_2}$ as follows. For every commutative $A$-algebra
$B$ and every $x\in R\otimes_A B$, take 
\[(D_1\times D_2)_B(x)=D_{1,B}(x)D_{2,B}(x).\]
If $\Gamma$ is a group, $\rho_1:\Gamma\ra\GL_{d_1}(A)$, $\rho_2:\Gamma\ra
\GL_{d_2}(A)$ are representations and
$D_1=D_{\rho_1}$, $D_2=D_{\rho_2}$, then $D_1\times D_2=D_{\rho_1\oplus\rho_2}$.

\end{rmk}

\begin{ex}
Let $\Gamma$ be a group, and let $X$ be the underlying set of $\Gamma$.
Then the canonical surjective $A$-algebra morphism $A\{X\}\ra A[\Gamma]$
gives rise to an injective morphism of functors
$\Det_{\Gamma,d}\ra\Det_{A\{X\},d}$.

\end{ex}

\begin{thm}[Roby, see \cite{Roby} III.1 and \cite{Chenevier2} 1.6]
Let $R$ be an $A$-algebra and $d\in\Nat$. Then the functor $\Det_{R,d}$
is representable by the $A$-algebra $(\Gamma^d_A(R))^\ab$, where
$\Gamma^d_A(R)$ is the $A$-algebra of divided powers of order $d$ relative
to $A$ and $(-)^\ab$ denotes the abelianization.

\end{thm}

See Roby's papers (\cite{Roby} III.1 and \cite{Roby2} II) for the
definition of $\Gamma^d_A(M)$ if $M$ is an $A$-module, and the $A$-algebra
structure on $\Gamma^d_A(R)$ if $R$ is an $A$-algebra.

We will also write $\Det_{R,d}$ for the affine
scheme representing the functor $\Det_{R,d}$.

\begin{rmk}
The scheme $\Det_{R,d}$ is a contravariant functor of the $A$-algebra
$R$. Indeed, if
$u:R\ra S$ is a morphism of $R$-algebras, then,
for every commutative $A$-algebra $B$ and every $D\in\Det_{S,d}(B)=
\Mcal_A^d(S,B)$, the maps $u^*(D)_C:R\otimes_A C\ra B\otimes_A C$,
$x\mapsto D((u\otimes\id_C)(x))$, for $C$ an object of $\Cf_A$,
define an element of $\Mcal_A^d(R,B)=\Det_{R,d}(B)$.

In particular, the scheme $\Det_{A\{X\},d}$ (resp. $\Det_{\Gamma,d}$)
is a contravariant functor of the set $X$ (resp. the group $\Gamma$).

\end{rmk}

\begin{rmk}\label{rmk_BC}
For every commutative $A$-algebra $B$, we have a natural morphism
of schemes $\Det_{R\otimes_A B,d}\ra\Det_{R,d}\times_{\Spec A} \Spec B$. This is an
isomorphism by Theorem III.3 on page 262 of~\cite{Roby}.

\end{rmk}

Let $X$ be a set.
We recall Vaccarino's construction of the universal ring of
$\Det(\Z\{X\},d)$, following Chenevier's presentation in
\cite[1.15]{Chenevier2}. Let $F_X(d)=\Z[x_{i,j},\ x\in X, 1\leq i,j\leq d]$
be the ring of polynomials on the variables $x_{i,j}$ for $x\in X$ and
$i,j\in\{1,2,\ldots,d\}$. The \emph{generic matrices representation} is the
ring morphism
\[\rho^\univ:\Z\{X\}\ra M_d(F_X(d))\]
sending each $x\in X$ to the matrix $(x_{ij})_{1\leq i,j\leq d}$. This
defines a degree $d$ homogeneous multiplicative polynomial law
$D_{\rho^\univ}:\Z\{X\}\ra E_X(d)$, where
$E_X(d)$ is the subring of $F_X(d)$ generated by the coefficients of
the characteristic polynomials of the $\rho^\univ(p)$, $p\in\Z\{X\}$.

\begin{thm}[Vaccarino, see Theorem~1.15 of~\cite{Chenevier2}] 
The ring
$E_X(d)$ and $D_{\rho^\univ}\in\Det_{\Z\{X\},d}(E_X(d))$ represent the functor
$\Det_{\Z\{X\},d}$.

\end{thm}

Using results of Donkin, we deduce the following corollary.

\begin{coro}\label{cor_DetX}
The morphism $M_d^X\ra\Det_{\Z\{X\},d}$ sending $\rho$ to $D_\rho$
(see Example~\ref{ex_det}(1)) induces an isomorphism
$\Spec(\Of(M_d^X)^{\GL_{d}})\iso\Det_{\Z\{X\},d}$.

\end{coro}

\begin{proof}
Note that $M_d^X=\Spec(F_X(d))$, and that the morphism $M_d^X\ra\Det_{\Z\{X\},d}$
of the statement corresponds to the multiplicative polynomial law
$\Z\{X\}\flnom{D_{\rho^\univ}}E_X(d)\subset F_X(d)$. On the other hand,
by the results of Donkin on the generators of $\Of(M_d^X)^{\GL_{d}}$
(see \cite[3.1]{Donkin}), we have $E_X(d)=\Of(M_d^X)^{\GL_{d}}$. The
corollary follows.

\end{proof}

\section{Lafforgue's definition of pseudocharacters}
\label{LPC}

Let $A$ be a commutative unital ring and $G$ be an affine group scheme
over $A$.

As in Section~\ref{Det}, we denote by $\Cf_A$ the category
of commutative unital $A$-algebras. If $B$ is a commutative $A$-algebra and
$X$ is a set, we write $\Cf(X,B)$ for the $B$-algebra of functions $X\ra B$.

For every set $X$ (resp. integer $n\in\Nat$), we make $G$ act on the scheme
$G^X$ (resp. $G^n$) by diagonal conjugation, and we denote by $\Of(G^X)^G$
(resp. $\Of(G^n)^G$) the $A$-algebra of
$G$-invariant regular functions on $G^X$ (resp. $G^n$).

\begin{defi}[See \cite{Lafforgue1} Proposition~11.7 and \cite{BHKT}
Definition~4.1]
\label{def_LPC}
Let $X$ be a set and $B$ be an object of $\Cf_A$. A 
\emph{$B$-valued
$G$-pseudocharacter for the set $X$} is a family of
$A$-algebra morphisms $\Of(G^n)^G\ra\Cf(X,B)$, for $n\geq 1$, satisfying
the following condition:
\begin{itemize}
\item[\setword{(LPC1)}{LPC1}]
For all $n,m\geq 1$, every map $\zeta:\{1,2,\ldots,m\}\ra\{1,2,\ldots,n\}$
and every $f\in\Of(G^m)^G$, if we define $f^\zeta\in\Of(G^n)^G$ by
$f^\zeta(g_1,\ldots,g_n)=f(g_{\zeta(1)},\ldots,g_{\zeta(m)})$, then we have
\[\Theta_m(f)(x_{\zeta(1)},\ldots,x_{\zeta(m)})=\Theta_n(f^\zeta)(x_1,\ldots,x_n)\]
for all $x_1,\ldots,x_n\in X$.

\end{itemize}

Suppose that $X$ is the underlying set of a group $\Gamma$.
Then a \emph{$B$-valued
$G$-pseudocharacter for the group $\Gamma$} is 
a $B$-valued
$G$-pseudocharacter for the set $X$ satisfying the following
additional condition:
\begin{itemize}
\item[\setword{(LPC2)}{LPC2}] 
For all $n\geq 1$ and $f\in\Of(G^n)^G$, if we define $\widehat{f}\in
\Of(G^{n+1})^G$ by $\widehat{f}(g_1,\ldots,g_{n+1})=f(g_1,\ldots,g_{n-1},
g_ng_{n+1})$, then we have
\[\Theta_{n+1}(\widehat{f})(\gamma_1,\ldots,\gamma_{n+1})=
\Theta_n(f)(\gamma_1,\ldots,\gamma_{n-1},\gamma_n\gamma_{n+1})\]
for all $\gamma_1,\ldots,\gamma_{n+1}\in\Gamma$.

\end{itemize} 

\end{defi}

We denote by $\LPC^1_{X,G}$ (resp. $\LPC_{\Gamma,G}$) the functor
$\Cf_A\ra\Set$ sending $B$ to the set of $B$-valued $G$-pseudocharacters 
for the set $X$ (resp. for the group $\Gamma$).
If $X$ is the underlying set of a group $\Gamma$, then
$\LPC_{\Gamma,G}$ is a subfunctor of $\LPC^1_{X,G}$.

\begin{ex}\label{ex_LPC}
\begin{itemize}
\item[(1)] If $X$ is a set and $\rho:X\ra G(A)$ is a map, then we define
an $A$-valued $G$-pseudocharacter $\Theta^1_\rho$ for the set $X$ by
\[\Theta^1_{\rho,n}(f)(x_1,\ldots,x_n)=f(\rho(x_1),\ldots,\rho(x_n)),\]
for every $n\geq 1$, every $f\in\Of(G^n)^G$ and all $x_1,\ldots,x_n\in X$.
Note that $\Theta^1_\rho$ only depends on the conjugacy class of $\rho$.

\item[(2)] Similarly,
if $\Gamma$ is a group and $\rho:\Gamma\ra G(A)$ is a morphism of groups,
then we define an $A$-valued $G$-pseudocharacter $\Theta_\rho$ 
for the group $\Gamma$ by
\[\Theta_{\rho,n}(f)(\gamma_1,\ldots,\gamma_n)=f(\rho(\gamma_1),\ldots,
\rho(\gamma_n)),\]
for every $n\geq 1$, every $f\in\Of(G^n)^G$ and all $\gamma_1,\ldots,\gamma_n
\in\Gamma$.
Again, $\Theta_\rho$ only depends on the conjugacy class of $\rho$.

\end{itemize}
\end{ex}

\begin{rmk}
The functor $\LPC^1_{X,G}$ depends contravariantly on the set $X$. Indeed, if
$u:X\ra Y$ is a map, then, for every commutative $A$-algebra $B$ and every
$\Theta=(\Theta_n)_{n\geq 1}\in\LPC^1_{Y,G}(B)$, the family
$u^*(\Theta)=(u^*(\Theta)_n)_{n\geq 1}$ defined by
\[u^*(\Theta)_n(f)(x_1,\ldots,x_n)=\Theta_n(f)(u(x_1),\ldots,u(x_n))\]
for $n\geq 1$, $f\in\Of(G^n)^G$ and $x_1,\ldots,x_n\in X$ is a
$B$-valued $G$-pseudocharacter for the set $X$.

Similary, the functor $\LPC_{\Gamma,G}$ depends contravariantly on
the group $\Gamma$.

\end{rmk}

If $X$ is a set, $x_1,\ldots,x_n\in X$ and $f\in\Of(G^n)$, we define a
regular function $f_{x_1,\ldots,x_n}\in\Of(G^X)$ by
\[f_{x_1,\ldots,x_n}((g_x)_{x\in X})=f(g_{x_1},\ldots,g_{x_n}).\]
Note that $f\in\Of(G^n)^G$ if and only if $f_{x_1,\ldots,x_n}\in\Of(G^X)^G$.

\begin{prop}\label{prop_rep_LPC1}
Let $X$ be a set and let $R^1_{X,G}=\Of(G^X)^G$. Consider the element
$\Theta^\univ$ of $\LPC^1_{X,G}(R^1_{X,G})$ defined by
\[\Theta^\univ_n(f)(x_1,\ldots,x_n)=f_{x_1,\ldots,x_n}\in R^1_{X,G},\]
for every $n\geq 1$, every $f\in\Of(G^n)^G$ and all $x_1,\ldots,x_n\in X$.
Then $R^1_{X,G}$ represents the functor $\LPC^1_{X,G}$, and $\Theta^\univ\in\LPC^1_{X,G}
(R^1_{X,G})$ is the universal element.

\end{prop}

\begin{proof}
The morphism of functors $\Spec(R^1_{X,G})\ra\LPC^1_{X,G}$ corresponding
to $\Theta^\univ$ sends a morphism of rings
$u:R^1_{X,G}\ra B$ to the $G$-pseudocharacter $\Theta_u$ defined by
\[\Theta_{u,n}(f)(x_1,\ldots,x_n)=u(\Theta^\univ_n(f)(x_1,\ldots,x_n)),\]
for $n\geq 1$, $f\in\Of(G^n)^G$ and $x_1,\ldots,x_n\in X$.
We check that it is an isomorphism. 
The central point is that, by definition of the product $G^X$, we have
$\Of(G^X)=\varinjlim_{Y\subset X\ \mathrm{finite}}\Of(G^Y)$; in particular, every
element of $\Of(G^X)^G$ is of the form $f_{x_1,\ldots,x_n}$, for some $n\geq 1$,
$x_1,\ldots,x_n\in X$ and $f\in\Of(G^n)^G$.

Let $B$ be a commutative $A$-algebra. Let $u,v:R^1_{X,G}\ra B$ be two morphisms of
$A$-algebras such that $\Theta_u=\Theta_v$. Let $h\in R^1_{X,G}$, and choose
an integer $n\geq 1$, $x_1,\ldots,x_n\in X$ and $f\in\Of(G^n)^G$ such that
$h=f_{x_1,\ldots,x_n}$. Then
\[u(h)=\Theta_{u,n}(f)(x_1,\ldots,x_n)=\Theta_{v,n}(f)(x_1,\ldots,x_n)=v(h).\]
This proves that $u=v$. Now let $\Theta\in\LPC^1_{X,G}(B)$; we want to find
a morphism of $A$-algebras $w:R^1_{X,G}\ra B$ such that $\Theta=\Theta_w$.
Let $Y$ be a finite subset of $X$. If $\zeta:Y\iso\{1,2,\ldots,n\}$ is
a bijection, then we get an isomorphism of $A$-algebras
$\Of(G^Y)^G\iso\Of(G^n)^G$ sending $h\in\Of(G^Y)^G$ to the regular function
$h^\zeta:(g_1,\ldots,g_n)\mapsto h((g_{\zeta(y)})_{y\in Y})$, and we define
$w_Y:\Of(G^Y)^G\ra B$ by
\[w_Y(h)=\Theta_n(h^\zeta)(\zeta^{-1}(1),\ldots,\zeta^{-1}(n)).\]
By condition \ref{LPC1}, this does not depend on the choice of $\zeta$ and,
if $Y\subset Z$ are finite subsets of $X$, then
$w_{Z\mid\Of(G^Y)^G}=w_Y$. So the family $(w_Y)_{Y\subset X\ \mathrm{finite}}$
defines a morphism of $A$-algebras $w:R^1_{X,G}\ra B$, and it follows immediately
from the definition of $\Theta^\univ$ that we have $\Theta=\Theta_w$.

\end{proof}

Let $\Gamma$ be a group.
For all $\gamma,\delta\in\Gamma$, we define a $G$-equivariant morphism
of $A$-modules $\varphi_{\gamma,\delta}:\Of(G\times G^\Gamma)\ra\Of(G^\Gamma)$
by
\[\varphi_{\gamma,\delta}(f)((g_\alpha)_{\alpha\in\Gamma})=f(g_{\gamma\delta},
(g_\alpha)_{\alpha\in\Gamma})-f(g_\gamma g_\delta,(g_\alpha)_{\alpha\in\Gamma}).\]
We have a morphism of $A$-algebras $\Of(G^\Gamma)\ra\Of(G\times G^\Gamma)$ induced
by the projection of $G\times G^\Gamma$ on its second factor, and this
morphism is equivariant for the action of $G$ by diagonal conjugation.
The map $\varphi_{\gamma,\delta}$ becomes $\Of(G^\Gamma)$-linear for this action
of $\Of(G^\Gamma)$ on $\Of(G\times G^\Gamma)$. In particular,
$\varphi_{\gamma,\delta}(\Of(G\times G^\Gamma))$ (resp.
$\varphi_{\gamma,\delta}(\Of(G\times G^\Gamma)^G)$) is an ideal of
$\Of(G^\Gamma)$ (resp. $\Of(G^\Gamma)^G$).

\begin{prop}
\begin{enumerate}
\item Let $J_{\Gamma,G}\subset\Of(G^\Gamma)$ be the sum of the images
of the $\varphi_{\gamma,\delta}$, for all $\gamma,\delta\in\Gamma$. 
Then $J_{\Gamma,G}$
is an ideal of $\Of(G^\Gamma)$, and $\Spec(\Of(G^\Gamma)/J_{\Gamma,G})
\subset\Spec(\Of(G^\Gamma))=G^\Gamma$
sends any commutative $B$-algebra $A$ to the set of maps
$\rho:\Gamma\ra G(B)$ that are morphisms of groups.

\item Let $I_{\Gamma,G}\subset J_{\Gamma,G}^G$
be the sum of the $\varphi_{\gamma,\delta}(\Of(G\times G^\Gamma)^G)$, 
for all $\gamma,\delta\in\Gamma$. Then $I_{\Gamma,G}$
is an ideal of $R^1_{\Gamma,G}=\Of(G^\Gamma)^G$, and, if we set
$R_{\Gamma,G}=R^1_{\Gamma,G}/I_{\Gamma,G}$, then the isomorphism
$\Spec(R^1_{\Gamma,G})\iso\LPC^1_{\Gamma,G}$ of Proposition~\ref{prop_rep_LPC1}
induces an isomorphism between the closed subscheme $\Spec(R_{\Gamma,G})$ of
$\Spec(R^1_{\Gamma,G})$ and $\LPC_{\Gamma,G} \subset \LPC^1_{\Gamma,G}$.

\end{enumerate}
\end{prop}

\begin{proof}
We already know that $J_{\Gamma,G}$ and $I_{\Gamma,G}$ are ideals, because
they are sums of ideals.

Let $B$ be a commutative $A$-algebra and $\rho:\Gamma\ra G(B)$ be a map;
this corresponds to a morphism of $A$-algebras $u:\Of(G^\Gamma)\ra B$, and
we have
\[u(f)=f((\rho(\gamma))_{\gamma\in\Gamma})\]
for every $f\in\Of(G^\Gamma)$.
We want to prove that $\rho$ is a morphism of groups if and only if
$u(J_{\Gamma,G})=0$. Suppose that $\rho$ is a morphism of groups. Let
$\gamma,\delta\in\Gamma$ and $f\in\Of(G\times G^\Gamma)$. Then
\begin{align*}
u(\varphi_{\gamma,\delta}(f))=
f(\rho(\gamma\delta),(\rho(\alpha)_{\alpha\in\Gamma}))-
f(\rho(\gamma)\rho(\delta),(\rho(\alpha)_{\alpha\in\Gamma}))=0
\end{align*}
as $\rho$ is a morphism.
This shows that $J_{\Gamma,G}\subset\Ker u$. Conversely, suppose that
$J_{\Gamma,G}\subset\Ker u$. Let $\gamma,\delta\in\Gamma$. Then, for every
$f\in\Of(G)$, we have
\[0=u(\varphi_{1,(\gamma,\delta)}(f))=f(\rho(\gamma\delta))-
f(\rho(\gamma)\rho(\delta)),\]
hence $\rho(\gamma\delta)=\rho(\gamma)\rho(\delta)$.
This finishes the proof of (i).

We now prove the second statement of (ii). 
Let $B$ be a commutative $A$-algebra, let
$\Theta\in\LPC^1_{\Gamma,G}(B)$, and let $u:R^1_{\Gamma,G}\ra B$ be the morphism
of $A$-algebras corresponding to $\Theta$. 
Suppose that $\Theta$ satisfies condition \ref{LPC2}. Let $\gamma,\delta\in
\Gamma$ and $f\in\Of(G\times G^\Gamma)^G$. Choose a finite subset
$\{\gamma_1,\ldots,\gamma_n\}$ of $\Gamma$ and $h\in\Of(G^{n+1})^G$ such
that
\[f(g,(g_\alpha)_{\alpha\in\Gamma})=h(g_{\gamma_1},\ldots,g_{\gamma_n},g).\]
Then we have
\[f(g_\gamma g_\delta,(g_\alpha)_{\alpha\in\Gamma})=\widehat{h}((g_\alpha)_{\alpha\in
\Gamma},g_\gamma,g_\delta),\]
so
\begin{align*}
u(\varphi_{\gamma,\delta}(f))  =
\Theta_{n+1}(h)(\gamma_1,\ldots,\gamma_n,\gamma\delta)-
\Theta_{n+2}(\widehat{h})(\gamma_1,\ldots,\gamma_n,\gamma,\delta)=0.
\end{align*}
This proves that $I_{\Gamma,G}\subset\Ker u$.

Conversely, suppose that $I_{\Gamma,G}\subset\Ker u$.
Let $n\geq 1$, $h\in\Of(G^n)^G$
and $\gamma_1,\ldots,\gamma_{n+1}\in\Gamma$. 
Define $f\in\Of(G\times G^\Gamma)^G$ by
\[f(g,(g_\alpha)_{\alpha\in\Gamma})=h(g_{\gamma_1},\ldots,g_{\gamma_{n-1}},g).\]
Then
\[0=u(\varphi_{\gamma_n,\gamma_{n+1}}(f))=
\Theta_n(h)(\gamma_1,\ldots,\gamma_{n-1},\gamma_n\gamma_{n+1})-
\Theta_n(\widehat{h})(\gamma_1,\ldots,\gamma_{n-1},\gamma_n,\gamma_{n+1}).\]
This implies that 
$\Theta$ satisfies condition \ref{LPC2}.

\end{proof}

Next we discuss the behavior of the functors $\LPC^1_{X,G}$ and
$\LPC_{\Gamma,G}$ under change of the base ring $A$.

We will use the notion of \emph{adequate homeomorphism} defined by
Alper (see \cite[Definition~3.3.1]{Alper}): a morphism of schemes
is an adequate homeomorphism if it is integral, a universal homeomorphism and
a local isomorphism at all points whose residue field is of characteristic
$0$. We will also use Alper's notion of \emph{geometrically reductive group
schemes}, see Definition~9.1.1 of~\cite{Alper}), which generalizes
that of reductive group schemes. In particular, if $G$ is an affine
smooth algebraic group over a field $k$, then it is geometrically
reductive if and only if it is reductive (Lemma~9.2.8 of~\cite{Alper}), and,
if $G\ra S$ is a smooth group scheme with connected fibers, then it
is a geometrically reductive group scheme if and only if it is reductive
(Theorem~9.7.5 of~\cite{Alper}).

Fix a commutative ring $A$ and a flat affine group scheme $G$ over $A$.
Let $B$ be a commutative $A$-algebra. For every $A$-module $V$ with an action
of $G$, we have $V^G\otimes_A B\subset (V\otimes_A B)^{G_B}$. As
$\Of(G^X)\otimes_A B=\Of(G_B^X)$ for every set $X$,
we deduce that $J_{\Gamma,G_B}=J_{\Gamma,G}\otimes_A B$ and
$I_{\Gamma,G_B}\supset I_{\Gamma,G}\otimes_A B$.
So we get morphisms of $B$-algebras, for $X$ a set and $\Gamma$ a group,
\[R^1_{X,G}\otimes_A B=\Of(G^X)^G\otimes_A B\ra\Of(G_B^X)^{G_B}=R^1_{X,G_B}\]
and
\[R_{\Gamma,G}\otimes_A B=
(\Of(G^\Gamma)^G\otimes_A B)/(I_{\Gamma,G}\otimes_A B)
\ra R_{\Gamma,G_B}\]
and corresponding morphisms of schemes
\[\beta_X:\LPC^1_{X,G_B}\ra\LPC^1_{X,G}\otimes_A B\]
and
\[\beta_\Gamma:\LPC_{\Gamma,G_B}\ra\LPC_{\Gamma,G}\otimes_A B.\]

\begin{prop}\label{prop_BC}
\begin{enumerate}
\item The morphisms $\beta_X$ and $\beta_\Gamma$ are isomorphisms 
in the following cases:
\begin{itemize}
\item[(a)] $B$ is a flat $A$-algebra;
\item[(b)] $A$ is a Dedekind domain and
$G$ is geometrically reductive over $A$ and has connected geometric fibers.
\end{itemize}

\item If $G$ is geometrically reductive over $A$ and has connected geometric fibers, then $\beta_X$ and
$\beta_\Gamma$ are always adequate homeomorphisms.

\end{enumerate}

\end{prop}

\begin{proof}
Point (ii) follows from Proposition~5.2.9(3) of~\cite{Alper}.
We prove (i). Suppose first that $B$ is flat over $A$. Then,
by Lemma~2 of Seshadri's paper~\cite{Seshadri}, for every
$A$-module $V$ with an action of $G$, the canonical morphism
$V^G\otimes_A B\ra(V\otimes_A B)^{G_B}$ is an isomorphism. Applying
this lemma to $\Of(G^X)$, we see that $\beta_X$ is an isomorphism.
As the functor $(-)\otimes_A B$ is right exact, Seshadri's lemma
also implies that $I_{\Gamma,G}\otimes_A B=I_{\Gamma,G_B}$, hence
that $\beta_\Gamma$ is an isomorphism.
We finally assume that $A$ is a principal ideal domain and that
$G$ is geometrically reductive over $A$. By Lemma~\ref{lemme_no_way},
the morphism $\Of(G^Y)^G\otimes_A B\ra\Of(G_B^Y)^{G_B}$ is an isomorphism
for every set $Y$; as in the proof of (i)(a), we conclude that
$\beta_X$ and $\beta_\Gamma$ are isomorphisms.

\end{proof}

\begin{lemma}\label{lemme_no_way}
Let $A$ be a Dedekind domain and $G$ be a geometrically reductive
group scheme with connected geometric fibers over $A$. Then, for every set $Y$, the injective morphism
$\Of(G^Y)^G\otimes_A B\ra\Of(G_B^Y)^{G_B}$ is an isomorphism.

\end{lemma}

\begin{rmk}
Lemma \ref{lemme_no_way} and its application to the functor $\LPC$ are
probably well-known to many people in some form. We found version of it
in a note by Chen (see~\cite{Chen}), as well as in a preprint by
Quast (see Section~1.3 of~\cite{Quast}).
  
\end{rmk}

\begin{proof}[Proof of Lemma \ref{lemme_no_way}]
We will use a number of results about the cohomology of algebraic groups
that are gathered in Jantzen's book \cite{Jantzen}; the particular
results that we need are due to Donkin and Mathieu, 
see~\cite{Jantzen} for the original references.

First, as the tensor product and the functor of invariants commute
with direct limits and as $\Of(G^Y)$ is the direct limit of the
$\Of(G^Z)$ for $Z \subset Y$ finite, we may assume that
$Y$ is finite.
By the universal coefficients theorem (Proposition~4.18 in Chapter~I
of~\cite{Jantzen}), we have an exact sequence
\[0\ra\Of(G^Y)^G\otimes_A B\ra\Of(G_B^Y)^{G_B}\ra
\Tor_1^A(\H^1(G,\Of(G^Y)),B)\ra 0.\]
So it suffices to show that $\H^1(G,\Of(G^Y))=0$. 
For this, it suffices to proves that the localization of
$\H^1(G,\Of(G^Y))$ at every ideal of $A$ is zero; as localizations are flat,
by the universal coefficient theorem again, we may replace $A$ by one of
its localizations, hence assume that $A$ is a field or a discrete valuation
ring. Now, by Lemma~B.9
of~\cite{Jantzen}, it suffices to prove that, for every maximal
ideal $\mfrak$ of $A$, if $k=A/\mfrak$, then  $\Of(G_k^Y)$ has a good
filtration (in the sense of II.4.16 of~\cite{Jantzen}) as a $G_k$-module.
This follows immediately from Propositions~4.20 and~4.21 of
Chapter~II of~\cite{Jantzen}.

\end{proof}

We finally investigate the relationship between pseudo-characters and
the character variety. 

\begin{defi}\label{def_carvar}
Let $A$ be a commutative ring, $G$ be a flat affine
group scheme over $A$ and $\Gamma$ be a group.
The \emph{$G$-character variety of $\Gamma$} is the affine scheme 
$\Char_{\Gamma,G}=\Spec((\Of(G^\Gamma)/J_{\Gamma,G})^G)$.

\end{defi}

If $B$ is an $A$-algebra, then we have
$I_{\Gamma,G_B}\supset I_{\Gamma,G}\otimes_A B$,
so we get a morphism of $B$-algebras
\[(\Of(G^\Gamma)/J_{\Gamma,G})^G\otimes_A B\ra
((\Of(G^\Gamma)/J_{\Gamma,G})\otimes_A B)^{G_B}=
(\Of(G_B^\Gamma)/I_{\Gamma,G_B})^{G_B},\]
and a corresponding morphism of schemes
\[\beta_\Gamma':\Char_{\Gamma,G_B}\ra\Char_{\Gamma,G}\otimes_A B.\]

\begin{lemma}\label{lemme_BC}
Suppose that $G$ is geometrically reductive over $A$ and has connected geometric fibers.
\begin{enumerate}
\item If $B$ is a flat $A$-algebra, then $\beta'_\Gamma$ is an
isomorphism.
\item In general, $\beta'_\Gamma$ is an adequate homeomorphism.

\end{enumerate}
\end{lemma}

\begin{proof}
Point (i) follows from Proposition~5.2.9(1) of~\cite{Alper}, and
point (ii) from Proposition~5.2.9(3) of the same paper.

\end{proof}

\begin{prop}\label{prop_charvar}
Suppose that $G$ is a geometrically reductive group scheme with connected geometric fibers over $A$
and that $A$ is a Dedekind domain or a field.
Let $\iota$ be the morphism $\Char_{\Gamma,G}\ra\LPC_{\Gamma,G}$ induced by
$\Of(G^\Gamma)^G/I_{\Gamma,G}\twoheadrightarrow\Of(G^\Gamma)^G/J_{\Gamma,G}^G
\hookrightarrow(\Of(G^\Gamma)/J_{\Gamma,G})^G$.
\begin{enumerate}
\item If $A$ is a field of characteristic $0$, then $\iota$ is an
isomorphism.
\item In general, $\iota$ is an adequate homeomorphism. 
\end{enumerate}

\end{prop}

The statement of the proposition is very close to that of
Proposition~11.7 of~\cite{Lafforgue1} and Theorem~4.5 of~\cite{BHKT}, and
its proof uses the same kind of ideas. We still include it here because
it is not very hard.

\begin{proof}
We prove (i). As $A$ is a field, the group $G$ is reductive over $A$. So,
for every algebraic representation $V$ of $G$ over $A$ (not
necessarily finite-dimensional), we have the Reynolds operator
$E=E_V:V\ra V$ (see \cite{Mumford} Definition 1.5), 
which is a $G$-equivariant projection with image $V^G$, compatible with
any morphism of representations $V\ra W$; also, if $V$ is a $A$-algebra
and the action of $G$ preserves its multiplication, then
$E_V$ is $V^G$-linear.
We claim that $I_{\Gamma,G}=J_{\Gamma,G}^G$. We already know that $I_{\Gamma,G}
\subset J_{\Gamma,G}^G$.
Conversely, let $h\in J_{\Gamma,G}^G$. Write $h=\sum_{i=1}^n
\varphi_{\gamma_i,\delta_i}(f_i)$,
with $\gamma_i,\delta_i\in\Gamma$ and $f_i\in\Of(G\times G^\Gamma)$. As the
$\varphi_{\gamma_i,\delta_i}$ are $G$-equivariant morphisms, they are 
compatible with the Reynolds operators, so
\[h=E(h)=\sum_{i=1}^n\varphi_{\gamma_i,\delta_i}(E(f_i))\in I_{\Gamma,G}.\]
It remains to prove that the injective morphism $\Of(G^\Gamma)^G/J_{\Gamma,G}^G\ra
(\Of(G^\Gamma)/J_{\Gamma,G})^G$ is also surjective. Let $f\in\Of(G^\Gamma)$, and
suppose that the class of $f$ modulo $J_{\Gamma,G}$ is $G$-invariant.
As the canonical surjection $\Of(G^\Gamma)\ra\Of(G^\Gamma)/J_{\Gamma,G}$ is
$G$-equivariant, it is compatible with the Reynolds operators, so we
deduce that $f-E(f)\in J_{\Gamma,G}$, which means that $f+J_{\Gamma,G}$
is in the image of $\Of(G^\Gamma)^G/J_{\Gamma,G}^G$.

We prove (ii). We know that $\Char_{\Gamma,G}\ra
\Spec(\Of(G^\Gamma)^G/J_{\Gamma,G}^G)$ is an adequate homeomorphism by
Lemma~5.2.12 of \cite{Alper} (see Remark~5.2.14 of~\emph{loc. cit.}).
So it remains to prove that $\iota':\Spec(\Of(G^\Gamma)^G/J_{\Gamma,G}^G)\ra
\Spec(\Of(G^\Gamma)^G/I_{\Gamma,G})$ is an adequate homeomorphism. 
This morphism is a closed embedding, hence it is integral, universally
injective and universally closed.
If $\Frac(A)$ is of characteristic $0$, then $\iota'$ becomes an
isomorphism after we tensor it by $\Frac(A)$ by 
Proposition~\ref{prop_BC}(ii); 
otherwise, its source
and target have no point with residue field of characteristic $0$.
So it remains to show that $\iota'$ is surjective, which is equivalent
to the fact that $\iota$ is surjective.

Let $x$ be a point of $\Spec(\Of(G^\Gamma)^G/I_{\Gamma,G})$,
which corresponds to a morphism of $A$-algebras $u:\Of(G^\Gamma)^G/I(\Gamma,G)\ra
K$, with $K$ a field. We want to find an extension $L$ of $K$ and
a morphism of $A$-algebras $v:(\Of(G^\Gamma)/J_{\Gamma,G})^G\ra L$ such that
$v\circ{\iota}^*=u$. We may always enlarge $K$ and $L$, so
we may assume that they are algebraically closed. Then, by
Proposition~\ref{prop_BC}(ii) and Lemma~\ref{lemme_BC}(ii), 
$\Char_{\Gamma,G}$ and $\Char_{\Gamma,G_K}$ (resp. $\LPC_{\Gamma,G}$ and
$\LPC_{\Gamma,G_K}$) have the same points over any algebraically closed
extension of $K$, so
we may assume that $A=K$ (and so that $G$ is reductive over $K$).
By Theorem~5.13 of~\cite{Popp} (or Lemma~5.2.1 and Remark~5.2.2
of~\cite{Alper}), there exists an extension $L$ of
$K$ and a morphism of $K$-algebras $w:\Of(G^\Gamma)\ra L$ such that
$u$ is induced by $w_{\mid\Of(G^\Gamma)^G}$.
The morphism $w$ corresponds to an element $(g_\gamma)\in G^\Gamma(L)$, and
we denote by $H$ the closed subgroup of $G_L$ generated by the
set $\{g_\gamma,\ \gamma\in\Gamma\}$. We choose $w$ such that
the dimension of the maximal tori in $Z_G(H)$ is maximal.
We claim that $H$ is then strongly reductive
in $G$ in the sense of Definition~16.1 of Richardson's paper~\cite{Richardson},
that is, $H$ is not contained in any proper parabolic subgroup of
$Z_G(S)$, where $S$ is a maximal torus of $Z_G(H)$.
Indeed, let $\lambda$ be a cocharacter of $Z_G(S)$, and suppose that
$H$ is contained in $P(\lambda):=\{g\in Z_G(S)\mid \lim_{t\to 0}\lambda(t)g
\lambda(t)^{-1}\ \mathrm{exists}\}$. For every $\gamma\in\Gamma$, let
$g'_\gamma=\lim_{t\to 0}\lambda(t)g_\gamma\lambda(t)^{-1}\in G(L)$. Then 
$(g_\gamma)_{\gamma\in\Gamma}$ and $(g'_\gamma)_{\gamma\in\Gamma}$ have the same
image in $\Spec(\Of(G^\Gamma)^G)(L)$, so the morphism $w':\Of(G^\Gamma)\ra L$
corresponding to $(g'_\gamma)_{\gamma\in\Gamma}\in G^\Gamma(L)$
satisfies $w'_{\mid\Of(G^\Gamma)^G}=w_{\mid\Of(G^\Gamma)^G}$. On the other hand,
the closed subgroup $H'$ of $G$ generated by the $g'_\gamma$ is contained
in the centralizer of $\lambda$ in $Z_G(S)$, so $\lambda$ has to be central
in $Z_G(S)$, otherwise $Z_G(H')$, which contains the group generated by $S$
and the image of $\lambda$, would have a maximal torus of dimension
greater than $\dim(S)$. So $H$ is not contained in any proper parabolic
subgroup of $Z_G(S)$.
Also, as $H$ is a Noetherian scheme, for every big enough finite subset
$X$ of $\Gamma$, the closed subgroup of $G$ generated by
$\{g_\gamma,\ \gamma\in X\}$ is equal to $H$. 

We now prove that, for this choice of $w$, the map
$\gamma\mapsto g_\gamma$ is a morphism
of groups; this implies that $w$ extends to $\Of(G^\Gamma)/J_{\Gamma,G}$,
hence defines a point $y$ of $\Char_{\Gamma,G}(L)$ such that $\iota(y)=x$.
Let $\gamma,\delta\in\Gamma$. Choose a finite subset $X$ of
$\Gamma$ such that $\gamma,\delta,\gamma\delta\in X$ and
the closed subgroup of $G$ generated by $\{g_\alpha,\ \alpha\in X\}$ is
equal to $H$.
As $w$ vanishes on $I_{\Gamma,G}$, the images of
$(g_{\gamma\delta},(g_\alpha)_{\alpha\in\Gamma})$ and
$(g_{\gamma}g_{\delta},(g_\alpha)_{\alpha\in\Gamma})$ by the map
$(G\times G^\Gamma)(L)\ra\Spec(\Of(G\times G^\Gamma)^G)(L)$ are
equal, so the images of
$(g_{\gamma\delta},(g_\alpha)_{\alpha\in X})$ and
$(g_{\gamma}g_{\delta},(g_\alpha)_{\alpha\in X})$ by the map
$(G\times G^X)(L)\ra\Spec(\Of(G\times G^X)^G)(L)$ are also equal. 
The closed subgroups of $G$ generated by the families
$(g_{\gamma\delta},(g_\alpha)_{\alpha\in X})$ and
$(g_{\gamma}g_{\delta},(g_\alpha)_{\alpha\in X})$ are both equal to $H$, hence
strongly reductive in $G$; so,
by Theorem~16.4 of Richardson's paper~\cite{Richardson},
the $G$-orbits of these families
in $(G\times G^X)(L)$ are closed. As they have the same image in
$\Spec(\Of(G\times G^X)^G)(L)$, and as $G\times G^X$ is of finite type
over $K$, this implies that they are in the same $G$-conjugacy class.
Let $h\in G(L)$ such that $h(g_{\gamma\delta},(g_\alpha)_{\alpha\in X})h^{-1}=
(g_{\gamma}g_{\delta},(g_\alpha)_{\alpha\in X})$. Then $h$ centralizes all the
$g_\alpha$ for $\alpha\in X$, so $g_\gamma g_\delta=hg_{\gamma\delta}h^{-1}=
g_{\gamma\delta}$. This finishes the proof.

\end{proof}

\begin{rmk}
The morphism $\Char_{\Gamma,G}\ra\LPC_{\Gamma,G}$ of Proposition~\ref{prop_charvar}
is an isomorphism if and only if both morphisms
$\Of(G^\Gamma)^G/I_{\Gamma,G}\twoheadrightarrow\Of(G^\Gamma)^G/J_{\Gamma,G}^G$ and
$\Of(G^\Gamma)^G/J_{\Gamma,G}^G
\hookrightarrow(\Of(G^\Gamma)/J_{\Gamma,G})^G$ are isomorphisms.
This seems unlikely, but we cannot offer a counterexample.

\end{rmk}

\section{A Lafforgue pseudocharacter gives rise to a determinant}
\label{LPCDet}

Let $A$ be a commutative unital ring and $d$ be a positive integer.
If $X$ is a set and $\Gamma$ is a group, we write $\LPC^1_{X,d}$ and
$\LPC_{\Gamma,d}$ instead of $\LPC^1_{X,\GL_{d,A}}$ and
$\LPC_{\Gamma,\GL_{d,A}}$.
We will also use the notation of Example~\ref{ex_LPC}.

Let $X$ be a set, $d$ be a positive integer,
$B$ be a commutative $A$-algebra and $\Theta=
(\Theta_n)$ be an element of $\LPC^1_{X,d}(B)$.
We want to construct an element $\alpha^1_X(\Theta)$ of
$\Det_{A\{X\},d}(B)$, that is, a degree $d$ homogeneous multiplicative
polynomial law from $A\{X\}$ to $B$ (seen as $A$-algebras).

Let $Y,Z$ be
sets and $\sigma:Y\ra Z$ be a map.
If $C$ is a commutative $A$-algebra, then we define a map
$\det_{\sigma,C}:C\{Y\}=A\{Y\}\otimes_A C\ra\Of({\GL}_{d,C}^Z)^{\GL_{d,C}}$ in the
following way. Let $p\in C\{Y\}$. Then $\det_{\sigma,C}(p)$ is the regular
function on ${\GL}_{d,C}^Z$ sending $(g_z)_{z\in Z}$ to
$\det(p((g_{\sigma(y)})_{y\in Y}))$ (where $\det$ is the usual determinant on
$\GL_d$), which is clearly $\GL_{d,C}$-invariant. Note that this construction
is functorial in $C$, and that we have $\Of(\GL_{d,C}^X)^{\GL_{d,C}}=
\Of({\GL}_{d,A}^Z)^{\GL_{d,A}}\otimes_A C$ by Lemma~\ref{lemme_no_way}.
Hence the family $(\det_{\sigma,C})_{C\in\Ob(\Cf_A)}$ defines a degree $d$
homogeneous multiplicative polynomial law from $A\{Y\}$ to
$\Of(\GL_{d,A}^Z)^{\GL_{d,A}}$, which we denote by $\det_\sigma$.

We come back to our $\Theta\in\LPC^1_{X,d}(B)$. By
Proposition~\ref{prop_rep_LPC1}, it corresponds to a morphism of
$A$-algebras $u_\Theta:\Of(\GL_{d,A}^X)^G\ra B$, and we send it to the
polynomial law
\[\alpha^1_X(\Theta)=u_\Theta\circ\det\nolimits_{\id_X}:A\{X\} \ra B.\]
In other words, for every commutative $A$-algebra $C$ and every
$p\in C\{X\}$, the element $\alpha^1_X(\Theta)_C(p)$ of
$B\otimes_A C$ is the image by $u_{\Theta}\otimes\id_C$ of the element
$(g_x)_{x\in X}\ra \det(p(g_x)_{x\in X})$ of $\Of(\GL_{d,C}^X)^{\GL_{d,C}}$.

The functoriality of $\alpha^1_X(\Theta)_C$ in $C$ follows immediately from
its definition, and the fact that it defines a degree $d$ homogeneous
multiplicative polynomial law follows from the properties of the determinant
on $\GL_d$.

\begin{prop}\label{prop_LPCDet}
\begin{enumerate}
\item Let $X$ be a set and $d\geq 1$. Then the maps
$\LPC^1_{X,d}(B)\ra\Det_{A\{X\},d}(B)=\Mcal_A(A\{X\},B)$, $\Theta\mapsto
\alpha^1_X(\Theta)$, form a morphism of functors
$\alpha^1_X:\LPC^1_{X,d}\ra\Det_{A\{X\},d}$ such that, for every
commutative $A$-algebra $B$ and every map $\rho:X\ra\GL_d(B)$, we have
$\alpha^1_X(\Theta^1_\rho)=D_\rho$.
Moreover, the morphisms $\alpha^1_X$ are natural in $X$.

\item Let $\Gamma$ be a group and $d\geq 1$. We denote by $X$ the
underlying set of $\Gamma$.
Then the morphism $\alpha_{X}:\LPC^1_{X,d}\ra\Det_{A\{X\},d}$
restricts to a morphism $\alpha_\Gamma:\LPC_{\Gamma,d}\ra\Det_{A[\Gamma],d}$
such that, for every
commutative $A$-algebra $B$ and every morphism of groups 
$\rho:\Gamma\ra\GL_d(B)$, we have
$\alpha_\Gamma(\Theta_\rho)=D_\rho$.
Moreover, the morphisms $\alpha_\Gamma$ are natural in $\Gamma$.

\end{enumerate}
\end{prop}

\begin{proof}
\begin{enumerate}
\item The fact that $\alpha^1_X$ is a morphism of functors and the naturality
in $X$ follow easily from the definition of $\alpha^1_X(\Theta)$.
Let $B$ be a commutative $A$-algebra and $\rho:X\ra\GL_d(B)$ be a map; then
the morphism of $A$-algebra $u:\Of(\GL_{d,A}^X)^{\GL_{d,A}} \ra B$ corresponding
to $\Theta_\rho$ is given by $u(f)=f((\rho(x))_{x\in X})$.
So, if $C$ is a commutative $A$-algebra and $p\in C\{X\}$. we have
\[\alpha^1_X(\Theta)_C(p)=\det(p((\rho(x)))_{x\in X})=\det(\rho(p))=D_\rho(p).\]


\item It suffices to prove that $\alpha_{X}$ sends $\LPC_{\Gamma,d}$
to $\Det_{A[\Gamma],d}$, the other statements then follow immediately from (i).

Let $B$ be a commutative $A$-algebra and $\Theta\in\LPC_{\Gamma,d}(B)$.
As $\Theta$ is also in $\LPC^1_{X,d}(B)$, we have a degree $d$
homogeneous polynomial law $D=\alpha^1_{X}(\Theta):A\{X\}\ra B$.
Saying that $D$ is in the image of $\Mcal_A(A[\Gamma],B)$ means that,
for every commutative $A$-algebra $C$,
the map $D_C:C\{X\}\ra B\otimes_A C$ factors through the obvious
surjection $\pi_C:C\{X\}\ra C[\Gamma]$. Fix a commutative $A$-algebra $C$.
We want to prove that, for all
$p,q\in C\{X\}$ such that $\pi_C(p)=\pi_C(q)$, we have $D_C(p)=D_C(q)$.
For $p\in C\{X\}$, let $n(p)$ be the sum over all the nonconstant
monomials $m$ appearing in $p$ of $\deg(m)-1$; so $n(p)=0$ if and only
if $p$ is of degree $\leq 1$. 
We claim that, for every $p\in C\{X\}$ such that
$n(p)\geq 1$, there exists $p_1\in C\{X\}$ such that $n(p_1)=n(p)-1$,
$\pi_C(p_1)=\pi_C(p)$ and $D_C(p_1)=D_C(p)$. This claim implies the desired
result; indeed, if $p,q\in C\{X\}$ are such that $\pi_C(p)=\pi_C(q)$,
then, by the claim, we can find $p_1,q_1\in C\{X\}$ such that
$n(p_1)=n(q_1)=0$, $\pi_C(p_1)=\pi_C(p)=\pi_C(q)=\pi_C(q_1)$,
$D_C(p_1)=D_C(p)$ and $D_C(q_1)=D_C(q)$; as $n(p_1)=n(q_1)=0$, the
polynomials $p_1$ and $q_1$ are of degree $\leq 1$, so $\pi_C(p_1)=\pi_C(q_1)$
implies that $p_1=q_1$, and then we have $D_C(p)=D_C(p_1)=D_C(q_1)=D_C(q)$.

So it suffices to prove the claim. Let $p\in C\{X\}$ such that
$n(p)\geq 1$, and let $m$ be a monomial of degree $\geq 2$ appearing
in $p$. Write $m=c x_{\gamma_1}\ldots x_{\gamma_k}$ with $c\in C\setminus\{0\}$,
$k\geq 2$ and $\gamma_1,\ldots,\gamma_k\in\Gamma$, where, for every
$\gamma\in\Gamma$, we denote by $x_\gamma$ the corresponding element of
$X$. Let $r=p-m$ and set $p_1=r+c x_{\gamma_1\gamma_2}x_{\gamma_3}\ldots x_{\gamma_k}$.
Then $n(p_1)=n(p)-1$ and $\pi_C(p_1)=\pi_C(p)$.
It remains to prove that $D_C(p_1)=D_C(p)$.
Let $u_\Theta:\Of(\GL_{d,A}^X)^{\GL_{d,A}}\ra B$ be the morphism
of $A$-algebras corresponding to the image of $\Theta$ in $\LPC^1_{\Gamma,d}(B)$.
We have
\[D_C(p)=\alpha_X^1(\Theta)_C(p)=u_\Theta(\det\nolimits_{\id_X}(p))\quad\mbox{and}
\quad D_C(p_1)=u_\Theta(\det\nolimits_{\id_X}(p_1).\]
As $\Theta$ is in $\LPC_{\Gamma,d}(B)$, the morphism $u_\Theta$ vanishes on
$I_{\Gamma,G}$. So, to show that $D_C(p)=D_C(p_1)$, it suffices to note
that $p_1-p=\varphi_{\gamma_1,\gamma_2}(F)$, with $F\in\Of(\GL_{d,A}\times
\GL_{d,A}^X)^{\GL_{d,A}}$ the function $(g,(g_\gamma)_{\gamma\in\Gamma})\mapsto
\det\left(r((g_\gamma)_{\gamma\in\Gamma})+c gg_{\gamma_3}\ldots g_{\gamma_k}\right)$.

\end{enumerate}
\end{proof}

\section{From determinants to pseudocharacters}
\label{DetLPC}

The main result of this note is the following theorem.

\begin{thm}\label{thm_main}
Let $A$ be a commutative ring and $d$ be a positive integer.
\begin{enumerate}
\item Let $X$ be a set. 
The morphism $\alpha^1_X:\LPC^1_{X,d}\ra
\Det(A\{X\},d)$ of Proposition~\ref{prop_LPCDet}(i)
corresponds by the isomorphisms $\LPC^1_{X,d}\simeq\Spec(\Of(\GL_d^X)^{\GL_d})$
and $\Det_{A\{X\},d}\simeq\Spec(\Of(M_d^X)^{\GL_d})$ of
Proposition~\ref{prop_rep_LPC1} and Corollary~\ref{cor_DetX}
to the restriction morphism $\Of(M_d^X)^{\GL_d}\ra\Of(\GL_d^X)^{\GL_d}$.
Moreover, $\alpha^1_X$ is injective, and, for every commutative $A$-algebra
$B$, an element $D$ of $\Det_{A\{X\},d}(B)$ is in the image of $\alpha^1_X$ if
and only if $D_B(x)\in B^\times$ for every $x\in X$.

\item For every set $\Gamma$, the morphism $\alpha_\Gamma:\LPC_{\Gamma,d}\ra
\Det(A[\Gamma],d)$ of Proposition~\ref{prop_LPCDet}(ii)
is an isomorphism.

\end{enumerate}
\end{thm}

\begin{proof}
By Remark~\ref{rmk_BC}
and Proposition~\ref{prop_BC}(i)(b), we may assume that $A=\Z$.

Let $X$ be a set.
As at the end of Section~\ref{Det}, we denote the universal matrices
representation by $\rho^\univ:X\ra M_d(F_X(d))$, where $F_X(d)=\Z[x_{i,j},\ 
x\in X,\ 1\leq i,j\leq d]$. The universal element of
$\Det_{\Z\{X\},d}$ is then $D^\univ=D_{\rho^\univ}:\Z\{X\}\ra E_X(d)$, where
$E_X(d)$ is generated by the coefficients of the characteristic polynomials
of the $\rho^\univ(p)$, $p\in\Z\{X\}$.

We prove (i).
The first statement follows from the fact that $\alpha^1_X(D_\rho)=\Theta_\rho$
for every commutative ring $A$ and every map $\rho:\Gamma\ra\GL_d(A)$.
For the injectivity of $\alpha^1_X$, it suffices to prove that
$\Of(\GL_d^X)^{\GL_d}$ is a localization of $\Of(M_d^X)^{\GL_d}$, but this
follows from the fact that $\Of(\GL_d^X)$ is the localization of
$\Of(M_d^X)$ by the multiplicative set generated by the functions
$\det_{x}:(g_y)_{y\in X}\mapsto\det(g_{x})$, $x\in X$, which are in
$\Of(M_d^X)^{\GL_d}$.
Finally, let $A$ be a commutative ring, let $D\in\Det_{\Z\{X\},d}(A)$, and
let $u:\Of(M_d^X)^{\GL_d}\ra A$ be the
morphism of rings corresponding to $D$. Then $D$ is the image of $\alpha^1_X$
if and only if $u$ extends to a morphism of rings $\Of(\GL_d^X)^{\GL_d}\ra A$, 
which is equivalent to the condition that $u(\det_x)\in A^\times$ for every
$x\in A$. As
$u(\det_x)=u(\det(\rho^\univ(x)))=D_A(x)$ for every $x\in X$, we get the
conclusion.

We prove (ii). Let $X$ be the underlying set of $\Gamma$.
We have a commutative diagram
\[\xymatrix{\LPC_{\Gamma,d}\ar[r]^-{\alpha_\Gamma}\ar@{^{(}->}[d] &
\Det_{\Z[\Gamma],d}\ar@{^{(}->}[d] \\
\LPC^1_{\Gamma,d}\ar[r]_-{\alpha^1_X} & \Det_{\Z\{X\},d}}\]
and $\alpha^1_X$ is injective by (i), so it suffices to check that the
image of $\alpha_\Gamma$ is $\Det_{\Z[\Gamma],d}$. For this, we check that
the universal element of $\Det_{\Z[\Gamma],d}$ is in the image of $\alpha_\Gamma$. 
Let
$\varphi:E_X(d)\simeq(\Gamma_\Z^d(\Z\{X\}))^\ab\ra A:=(\Gamma_\Z(\Z[\Gamma]))
^\ab$ be induced by the quotient map $\pi:\Z\{X\}\ra\Z[\Gamma]$. 
The universal element $D$ of $\Det_{\Z[\Gamma],d}$ is defined by
$D\circ\pi=\varphi\circ D^\univ$, and its preimage
by $\alpha^1_X$ is the element $\Theta$ of $\LPC^1_{\Gamma,d}$ defined
by
\[\Theta_n(f)(\gamma_1,\ldots,\gamma_n)=\varphi(\Theta^\univ_n(f)(\gamma_1,\ldots,
\gamma_n)),\]
where $\Theta^\univ=(\alpha^1_X)^{-1}(D^\univ)$.
It suffices to show that $\Theta$ satisfies
condition \ref{LPC2}; we will then have $\Theta\in\LPC_{\Gamma,d}(A)$ and
$\alpha_\Gamma(\Theta)=D$.
Let $n$ be a positive integer and
$\gamma_1,\ldots,\gamma_{n+1}\in\Gamma$. We want to check that
$\Theta_n(f)(\gamma_1,\ldots,\gamma_{n-1},\gamma_n\gamma_{n+1})=
\Theta_{n+1}(\widehat{f})(\gamma_1,\ldots,\gamma_{n+1})$ for every
$f\in\Of(\GL_d^n)^{\GL_d}$. As both sides are $\Z$-algebra morphisms in
$f$, it suffices to check it on generators of $\Of(\GL_d^n)^{\GL_d}$.
So, by results of Donkin (cf. \cite[3.1]{Donkin}), we may assume that
\[f(g_1,\ldots,g_n)=\Lambda_k(g_{i_1}\ldots g_{i_r}),\]
where $k\in\{1,\ldots,n\}$, $\Lambda_k$ is the $k$th coefficient of
the characteristic polynomial (i.e. $\Lambda_k(g)=(-1)^k\Tr(\Lambda^k g)$),
$r\in\Nat$ and $i_1,\ldots,i_r\in\{1,2,\ldots,n\}$.
For every $\gamma\in\Gamma$, we denote by $x_\gamma$ the element of $X$
corresponding to $\gamma$. For $i\in\{1,2,\ldots,n-1\}$,
we set $y_i=z_i=x_{\gamma_i}$, and we also set $y_n=x_{\gamma_n\gamma_{n+1}}$,
$z_n=x_{\gamma_n}x_{\gamma_{n+1}}$; these are all elements of $\Z\{X\}$.
We then have
\begin{align*}
\Theta_n(f)(\gamma_1,\ldots,\gamma_{n-1},\gamma_n\gamma_{n+1}) &=
\varphi(\Theta^\univ_n(f)(\gamma_1,\ldots,\gamma_{n-1},\gamma_n\gamma_{n+1}))\\
& = \varphi(f(\rho^\univ(x_{\gamma_1}),\ldots,\rho^\univ(x_{\gamma_{n-1}}),
\rho^\univ(x_{\gamma_n\gamma_{n+1}})))\\
& = \varphi(\Lambda_k(\rho^\univ(y_{i_1}\ldots y_{i_r})))
\end{align*}
and
\begin{align*}
\Theta_{n+1}(\widehat{f})(\gamma_1,\ldots,\gamma_{n+1}) &=
\varphi(\Theta^\univ_{n+1}(\widehat{f})(\gamma_1,\ldots,\gamma_{n+1}))\\
& = \varphi(\widehat{f}(\rho^\univ(x_{\gamma_1}),\ldots,\rho^\univ(x_{\gamma_{n+1}})))\\
& = \varphi(\Lambda_k(\rho^\univ(z_{i_1}\ldots z_{i_r}))).
\end{align*}
As $\pi\circ\det\circ\rho^\univ=
\varphi\circ D^\univ=D\circ\pi$ as polynomial laws, we get a commutative
diagram
\[\xymatrix{\Z\{X\}[t]\ar[r]^-{\det\circ\rho^\univ}
\ar[d]_{\pi\otimes\id_{\Z[t]}} &
E_X(d)[t]\ar[d]^{\varphi\otimes\id_{\Z[t]}} \\
\Z[\Gamma][t]\ar[r]_-{D} &
(\Gamma_\Z^d(\Z[\Gamma]))^\ab[t]
}\]
As $(\pi\otimes\id_{\Z[t]})(t-y_{i_1}\ldots y_{i_r})=
(\pi\otimes\id_{\Z[t]})(t-z_{i_1}\ldots z_{i_r})$, we get that
\begin{equation}\tag{**}\label{eq_det}
(\varphi\otimes\id_{\Z[t]})(\det\circ\rho^\univ(t-y_{i_1}\ldots y_{i_r}))
=(\varphi\otimes\id_{\Z[t]})(\det\circ\rho^\univ(t-z_{i_1}\ldots z_{i_r})).
\end{equation}
But, for every $p\in\Z\{X\}$, we have
\[\det\circ\rho^\univ(t-p)=t^d+\sum_{k=1}^d\Lambda_k(\rho^\univ(p))t^{d-k}
\in E_X(d)[t],\]
so identity \eqref{eq_det} implies that
\[\varphi(\Lambda_k(\rho^\univ(y_{i_1}\ldots y_{i_r})))=
\varphi(\Lambda_k(\rho^\univ(z_{i_1}\ldots z_{i_r}))),\]
as desired.

\end{proof}

\bibliographystyle{plain}
\bibliography{main}

\def\cprime{$'$}
\begin{thebibliography}{10}

\bibitem{Alper}
Jarod Alper.
\newblock Adequate moduli spaces and geometrically reductive group schemes.
\newblock {\em Algebr. Geom.}, 1(4):489--531, 2014.

\bibitem{BHKT}
Gebhard B{\"o}ckle, Michael Harris, Chandrashekhar Khare, and Jack~A. Thorne.
\newblock {{\(\hat{G}\)}}-local systems on smooth projective curves are
  potentially automorphic.
\newblock {\em Acta Math.}, 223(1):1--111, 2019.

\bibitem{Chen}
William Chen.
\newblock \textit{Two base change results for rings of invariants}.
\newblock
  \url{https://static1.squarespace.com/static/60d92f4398f4bb6a7b5a99c6/t/61651a3abe21161e3ed6e7e5/1634015802361/base_change.pdf},
  2021.

\bibitem{Chenevier1}
Ga\"etan Chenevier.
\newblock \textit{Repr\'esentations galoisiennes automorphes et cons\'equences
  arithm\'etiques des conjectures de Langlands et Arthur (th\`ese
  d'habilitation)}.
\newblock \url{http://gaetan.chenevier.perso.math.cnrs.fr/hdr/HDR.pdf}, 2013.

\bibitem{Chenevier2}
Ga{\"e}tan Chenevier.
\newblock The {{\(p\)}}-adic analytic space of pseudocharacters of a profinite
  group and pseudorepresentations over arbitrary rings.
\newblock In {\em Automorphic forms and Galois representations. Proceedings of
  the 94th London Mathematical Society (LMS) -- EPSRC Durham symposium, Durham,
  UK, July 18--28, 2011. Volume 1}, pages 221--285. Cambridge: Cambridge
  University Press, 2014.

\bibitem{Donkin}
Stephen Donkin.
\newblock Invariants of several matrices.
\newblock {\em Invent. Math.}, 110(2):389--401, 1992.

\bibitem{Jantzen}
Jens~Carsten Jantzen.
\newblock {\em Representations of algebraic groups}, volume 107 of {\em
  Mathematical Surveys and Monographs}.
\newblock American Mathematical Society, Providence, RI, second edition, 2003.

\bibitem{Lafforgue1}
Vincent Lafforgue.
\newblock Chtoucas pour les groupes r\'eductifs et param\'etrisation de
  langlands globale.
\newblock {\em J. Am. Math. Soc.}, 31(3):719--891, 2018.

\bibitem{Lafforgue2}
Vincent Lafforgue.
\newblock Shtukas for reductive groups and {Langlands} correspondence for
  function fields.
\newblock In {\em Proceedings of the international congress of mathematicians,
  ICM 2018, Rio de Janeiro, Brazil, August 1--9, 2018. Volume I. Plenary
  lectures}, pages 635--668. Hackensack, NJ: World Scientific; Rio de Janeiro:
  Sociedade Brasileira de Matem{\'a}tica (SBM), 2018.

\bibitem{Mumford}
D.~Mumford and J.~Fogarty.
\newblock {\em Geometric invariant theory. 2nd enlarged ed}, volume~34 of {\em
  Ergeb. Math. Grenzgeb.}
\newblock Springer-Verlag, Berlin, 1982.

\bibitem{Popp}
Herbert Popp.
\newblock {\em Moduli theory and classification theory of algebraic varieties},
  volume 620 of {\em Lect. Notes Math.}
\newblock Springer, Cham, 1977.

\bibitem{Quast}
Julian Quast.
\newblock \textit{Deformation of $G$-valued pseudocharacters}.
\newblock
  \url{http://www.julianquast.de/files/Deformations_of_G-valued_Pseudocharacters.pdf},
  2023.

\bibitem{Richardson}
R.~W. Richardson.
\newblock Conjugacy classes of n-tuples in {Lie} algebras and algebraic groups.
\newblock {\em Duke Math. J.}, 57(1):1--35, 1988.

\bibitem{Roby}
N.~Roby.
\newblock Lois polyn{\^o}mes et lois formelles en th{\'e}orie des modules.
\newblock {\em Ann. Sci. {\'E}c. Norm. Sup{\'e}r. (3)}, 80:213--348, 1963.

\bibitem{Roby2}
Norbert Roby.
\newblock Lois polyn{\^o}mes multiplicatives universelles.
\newblock {\em C. R. Acad. Sci., Paris, S{\'e}r. A}, 290:869--871, 1980.

\bibitem{Seshadri}
C.~S. Seshadri.
\newblock Geometric reductivity over arbitrary base.
\newblock {\em Adv. Math.}, 26:225--274, 1977.

\bibitem{Taylor}
Richard Taylor.
\newblock Galois representations associated to {Siegel} modular forms of low
  weight.
\newblock {\em Duke Math. J.}, 63(2):281--332, 1991.

\bibitem{Vaccarino}
Francesco Vaccarino.
\newblock Homogeneous multiplicative polynomial laws are determinants.
\newblock {\em J. Pure Appl. Algebra}, 213(7):1283--1289, 2009.

\bibitem{Wiles}
A.~Wiles.
\newblock On ordinary {{\(\lambda\)}}-adic representations associated to
  modular forms.
\newblock {\em Invent. Math.}, 94(3):529--573, 1988.

\end{thebibliography}

\end{document}